\pdfoutput=1
\RequirePackage{ifpdf}
\ifpdf 
\documentclass[pdftex]{sigma}
\else
\documentclass{sigma}
\fi

\numberwithin{equation}{section}

\newtheorem{Theorem}{Theorem}[section]
\newtheorem{Corollary}[Theorem]{Corollary}
{ \theoremstyle{definition}
\newtheorem{Example}[Theorem]{Example}
}

\newcommand{\mathd}{\mathrm{d}}

\begin{document}


\renewcommand{\thefootnote}{$\star$}

\newcommand{\arXivNumber}{1509.00886}

\renewcommand{\PaperNumber}{083}

\FirstPageHeading

\ShortArticleName{Certain Integrals Arising from Ramanujan's Notebooks}

\ArticleName{Certain Integrals Arising\\ from Ramanujan's Notebooks\footnote{This paper is a~contribution to the Special Issue
on Orthogonal Polynomials, Special Functions and Applications.
The full collection is available at \href{http://www.emis.de/journals/SIGMA/OPSFA2015.html}{http://www.emis.de/journals/SIGMA/OPSFA2015.html}}}

\Author{Bruce C.~{BERNDT}~$^\dag$ and Armin {STRAUB}~$^\ddag$}
\AuthorNameForHeading{B.C.~Berndt and A.~Straub}

\Address{$^\dag$~University of Illinois at Urbana--Champaign, 1409 W~Green St, Urbana, IL 61801, USA}
\EmailD{\href{mailto:berndt@illinois.edu}{berndt@illinois.edu}}

\Address{$^\ddag$~University of South Alabama, 411 University Blvd N, Mobile, AL 36688, USA}
\EmailD{\href{mailto:straub@southalabama.edu}{straub@southalabama.edu}}

\ArticleDates{Received September 05, 2015, in f\/inal form October 11, 2015; Published online October 14, 2015}

\Abstract{In his third notebook, Ramanujan claims that
$$ \int_0^\infty \frac{\cos(nx)}{x^2+1} \log x \,\mathd x + \frac{\pi}{2} \int_0^\infty \frac{\sin(nx)}{x^2+1} \mathd x = 0. $$
  In a following cryptic line, which only became visible in a recent reproduction of
  Ramanujan's notebooks, Ramanujan indicates that a similar relation exists
  if $\log x$ were replaced by $\log^2x$ in the f\/irst integral and $\log x$
  were inserted in the integrand of the second integral.  One of the goals of
  the present paper is to prove this claim by contour integration.  We further
  establish general theorems similarly relating large classes of inf\/inite
  integrals and illustrate these by several examples.}


\Classification{33E20}

\renewcommand{\thefootnote}{\arabic{footnote}}
\setcounter{footnote}{0}

\section{Introduction}

If you attempt to f\/ind the values of the integrals
\begin{gather}\label{1}
\int_0^{\infty}\frac{\cos(nx)}{x^2+1}\log x \,\mathd x \qquad\text{and}\qquad\int_0^{\infty}\frac{\sin(nx)}{x^2+1}\mathd x, \qquad n>0,
\end{gather}
by consulting tables such as those of Gradshteyn and Ryzhik \cite{gr} or by
invoking a computer algebra system such as \emph{Mathematica}, you will be
disappointed, if you hoped to evaluate these integrals in closed form, that
is, in terms of elementary functions.  On the other hand, the latter integral
above can be expressed in terms of the exponential integral $\operatorname{Ei}(x)$~\cite[formula~(3.723), no.~1]{gr}.
Similarly, if $1/(x^2+1)$ is replaced by any even rational function with the
degree of the denominator at least one greater than the degree of the
numerator, it does not seem possible to evaluate any such integral in closed
form.

However, in his third notebook, on p.~391 in the pagination of the second
volume of \cite{nb}, Ramanujan claims that the two integrals in \eqref{1} are
simple multiples of each other.  More precisely,
\begin{gather}\label{2}
\int_0^{\infty}\frac{\cos(nx)}{x^2+1}\log x \,
\mathd x +\frac{\pi}{2}\int_0^{\infty}\frac{\sin(nx)}{x^2+1}\mathd x =0.
\end{gather}
Moreover, to the left of this entry, Ramanujan writes, ``contour
integration''.  We now might recall a couple of sentences of G.H.~Hardy from
the introduction to Ramanujan's \emph{Collected papers} \cite[p.~xxx]{cp},
``\dots\ he had [before arriving in England] never heard of \dots\ Cauchy's
theorem, and had indeed but the vaguest idea of what a function of a complex
variable was''.  On the following page, Hardy further wrote, ``In a few
years' time he had a very tolerable knowledge of the theory of functions
\dots''.  Generally, the entries in Ramanujan's notebooks were recorded by
him in approximately the years 1904--1914, prior to his departure for
England.  However, there is evidence that some of the entries in his third
notebook were recorded while he was in England.  Indeed, in view of Hardy's
remarks above, almost certainly, \eqref{2}~is such an entry.  A~proof of~\eqref{2} by contour integration was given by the f\/irst author in his book
\cite[pp.~329--330]{IV}.

The identity \eqref{2} is interesting because it relates in a simple way two
integrals that we are unable to individually evaluate in closed form.
On the other hand, the simpler integrals
\begin{gather*}
\int_0^{\infty}\frac{\cos(nx)}{x^2+1}
\mathd x=\frac{\pi e^{-n}}{2} \qquad\text{and}\qquad \int_{-\infty}^{\infty}\frac{\sin(nx)}{x^2+1}\mathd x =0
\end{gather*}
have well-known and trivial evaluations, respectively.

 With the use of the most up-to-date photographic techniques, a new edition
 of \emph{Ramanujan's Notebooks}~\cite{nb} was published in~2012 to help
 celebrate the 125th anniversary of Ramanujan's birth.  The new reproduction
 is vastly clearer and easier to read than the original edition.  When the
 f\/irst author reexamined~\eqref{2} in the  new edition, he was surprised to
 see that Ramanujan made a further claim concerning~\eqref{2} that was not
 visible in the original edition of~\cite{nb}.  In a~cryptic one line, he
 indicated that a relation similar to~\eqref{2} existed if $\log x$ were
 replaced by $\log^2x$ in the f\/irst integral and $\log x$ were inserted in the
 integrand of the second integral of~\eqref{2}.  One of the goals of the
 present paper is to prove (by contour integration) this unintelligible entry
 in the f\/irst edition of the notebooks~\cite{nb}.   Secondly, we establish
 general theorems relating large classes of inf\/inite integrals for which
 individual evaluations in closed form are not possible by presently known
 methods.  Several further examples are given.

\section{Ramanujan's extension of~(\ref{2})}

We prove the entry on p.~391 of \cite{nb} that resurfaced with the new printing of~\cite{nb}.

\begin{Theorem}\label{theorem1} We have
\begin{gather}\label{4}
\int_0^{\infty}\frac{\sin(nx)}{x^2+1}\mathd x +\frac{2}{\pi}\int_0^{\infty}\frac{\cos(nx)}{x^2+1}\log x\,\mathd x =0
\end{gather}
and
\begin{gather}\label{4a}
\int_0^{\infty}\frac{\sin(nx)\log x}{x^2+1}\mathd x+\frac{1}{\pi}\int_0^{\infty}\frac{\cos(nx)\log^2x}{x^2+1}\mathd x=\frac{\pi^2e^{-n}}{8}.
\end{gather}
\end{Theorem}

\begin{proof}  Def\/ine a branch of $\log z$ by $-\frac12\pi<\theta=\arg z\leq
  \frac32\pi$.  We integrate
\begin{gather*}
f(z):=\frac{e^{inz}\log^2z}{z^2+1}
\end{gather*}
over the positively oriented closed contour $C_{R,\varepsilon}$  consisting
of the semi-circle $C_{R}$ given by \linebreak \mbox{$z=Re^{i\theta}$},
$0\leq\theta\leq\pi$,
the interval $[-R,-\varepsilon]$, the semi-circle $C_{\varepsilon}$ given by
$z=\varepsilon e^{i\theta}$, $\pi\geq\theta\geq0$, and the interval $[\varepsilon,R]$,
where $0<\varepsilon<1$ and $R>1$. On the interior of  $C_{R,\varepsilon}$
there is a simple pole at~$z=i$, and so by the residue theorem,
\begin{gather}\label{5}
\int_{C_{R,\varepsilon}}f(z)\mathd z=2\pi i\frac{e^{-n} \cdot\big({-}\frac14\pi^2\big)}{2i}=-\frac{e^{-n}\pi^3}{4}.
\end{gather}
Parameterizing the respective semi-circles, we can readily show that
\begin{gather}\label{6}
\int_{C_{\varepsilon}}f(z)\mathd z=o(1),
\end{gather}
as $\varepsilon\to0$, and
\begin{gather}\label{7}
\int_{C_{R}}f(z)\mathd z=o(1),
\end{gather}
as $R\to\infty$.
Hence, letting $\varepsilon\to0$ and $R\to\infty$ and combining \eqref{5}--\eqref{7}, we conclude that
 \begin{gather}\label{8}
 - \frac{e^{-n}\pi^3}{4}=\int_{-\infty}^0\frac{e^{inx}(\log|x|+i\pi)^2}{x^2+1}\mathd x+\int_0^{\infty}\frac{e^{inx}\log^2x}{x^2+1}\mathd x\\
 \hphantom{- \frac{e^{-n}\pi^3}{4}}{}
 =\int_0^{\infty}\!\frac{(\cos(nx)-i\sin(n x))(\log x+i\pi)^2}{x^2+1}\mathd x+\int_0^{\infty}\!\frac{(\cos(nx)+i\sin(nx))\log^2x}{x^2+1}\mathd x.\nonumber
 \end{gather}
If we equate real parts in \eqref{8}, we f\/ind that
\begin{gather}\label{boo}
-\frac{e^{-n}\pi^3}{4} =\int_0^{\infty}\frac{\cos(nx)\big(2\log^2x-\pi^2\big)+2\pi\sin(nx)\log x}{x^2+1}\mathd x.
\end{gather}
It is easy to show, e.g., by contour integration, that
\begin{gather}\label{easy}
\int_0^{\infty}\frac{\cos(nx)}{x^2+1}\mathd x=\frac{\pi e^{-n}}{2}.
\end{gather}
(In his quarterly reports, Ramanujan derived~\eqref{easy} by a~dif\/ferent method \cite[p.~322]{I}.)
Putting this evaluation in~\eqref{boo}, we readily deduce~\eqref{4a}.
If we equate imaginary parts in~\eqref{8}, we deduce that
\begin{gather*}
0=\int_0^{\infty}\frac{\pi^2\sin(nx)+2\pi\cos(nx)\log x}{x^2+1}\mathd x,
\end{gather*}
from which the identity \eqref{4} follows.
\end{proof}

\section{A second approach to the entry at the top of p.~391}

\begin{Theorem} \label{thm:int:s}
For $s \in (- 1, 2)$ and $n \geq 0$,
  \begin{gather}
    \frac{\pi}{2} e^{- n} = \int_0^{\infty} \frac{\cos (n x - \pi s / 2)}{x^2
    + 1} x^s \mathd x. \label{eq:int:s}
  \end{gather}
\end{Theorem}

Before indicating a proof of Theorem~\ref{thm:int:s}, let us see how the
integral~\eqref{eq:int:s} implies Ramanujan's integral relations~\eqref{4} and~\eqref{4a}. Essentially,
all we have to do is to take derivatives of \eqref{eq:int:s} with respect to~$s$ (and interchange the order of dif\/ferentiation and integration); then, upon
setting $s = 0$, we deduce~\eqref{4} and~\eqref{4a}.

First, note that upon setting $s = 0$ in \eqref{eq:int:s}, we obtain  \eqref{easy}.
On the other hand, taking a~derivative of~\eqref{eq:int:s} with respect to
$s$, and then setting $s = 0$, we f\/ind that
\begin{gather}
  0 = \int_0^{\infty} \frac{\cos (n x)}{x^2 + 1} \log x \,\mathd x +
  \frac{\pi}{2} \int_0^{\infty} \frac{\sin (n x)}{x^2 + 1} \mathd x,
  \label{eq:int:rama}
\end{gather}
which is the formula~\eqref{4} that Ramanujan recorded on p.~391. Similarly, taking
two derivatives of~\eqref{eq:int:s} and then putting $s=0$, we arrive at
\begin{gather*}
  0 = \int_0^{\infty} \frac{\cos (n x)}{x^2 + 1} \log^2 x\, \mathd x + \pi
   \int_0^{\infty} \frac{\sin (n x)}{x^2 + 1} \log x \,\mathd x -
   \frac{\pi^2}{4} \int_0^{\infty} \frac{\cos (n x)}{x^2 + 1} \mathd x,
\end{gather*}
which, using \eqref{easy}, simplif\/ies to
\begin{gather}\label{eq:int:rama2}
  \frac{\pi^3}{8} e^{- n} = \int_0^{\infty} \frac{\cos (n x)}{x^2 + 1} \log^2
   x \,\mathd x + \pi \int_0^{\infty} \frac{\sin (n x)}{x^2 + 1} \log x\,
   \mathd x.
\end{gather}
Note that this is  Ramanujan's previously unintelligible formula~\eqref{4a}. If we
likewise take $m$ derivatives before setting $s = 0$, we obtain the following
general set of relations connecting the integrals
\begin{gather*}
  I_m := \int_0^{\infty} \frac{\cos (n x)}{x^2 + 1} \log^m x \,\mathd x,
   \qquad J_m := \int_0^{\infty} \frac{\sin (n x)}{x^2 + 1} \log^m  x
   \,\mathd x.
\end{gather*}

\begin{Corollary}
  For $m \geq 1$,
  \begin{gather*}
    0 = \sum_{k = 0}^m \binom{m}{k} \left( \frac{\pi}{2} \right)^k (- 1)^{[k
     / 2]} \left\{ \begin{matrix}
       I_{m - k}, & \text{if $k$ is even}\\
       J_{m - k}, & \text{if $k$ is odd}
     \end{matrix} \right\} .
  \end{gather*}
\end{Corollary}

We now provide a proof of Theorem \ref{thm:int:s}.

\begin{proof}
  In analogy with our previous proof, we integrate
  \begin{gather*}
    f_s (z) := \frac{e^{i n z} z^s}{z^2 + 1}
  \end{gather*}
  over the contour $C_{R, \varepsilon}$ and let $\varepsilon \rightarrow 0$ and
  $R \rightarrow \infty$. Here, $z^s = e^{s \log z}$ with $-\frac12\pi<\arg z \leq \frac32\pi$, as above. By the residue theorem,
  \begin{gather}\label{aa}
    \int_{C_{R, \varepsilon}} f_s (z) \mathd z = 2 \pi i \frac{e^{- n} e^{\pi
     i s / 2}}{2 i} = \pi e^{- n} e^{\pi i s / 2} .
  \end{gather}
 Letting  $\varepsilon \rightarrow 0$ and $R \rightarrow \infty$, and using
 bounds for the integrand on the semi-circles as we did above, we deduce that
  \begin{gather}\label{bb}
    \lim_{\substack{R\to\infty\\\varepsilon\to0}}\int_{C_{R, \varepsilon}} f_s (z) \mathd z = \int_{- \infty}^{\infty}
     \frac{e^{i n x} x^s}{x^2 + 1} \mathd x = \int_0^{\infty} \frac{e^{- i n
     x} x^s e^{\pi i s}}{x^2 + 1} \mathd x + \int_0^{\infty} \frac{e^{i n x}
     x^s}{x^2 + 1} \mathd x.
  \end{gather}
  Combining \eqref{aa} and \eqref{bb}, we f\/ind that
  \begin{gather}
    \pi e^{- n} e^{\pi i s / 2} = \int_0^{\infty} \big(e^{i n x} + e^{- i n x} e^{\pi i s}\big)
    \frac{x^s }{x^2 + 1} \mathd x.\label{eq:int:s:i}
  \end{gather}
  We then divide both sides of \eqref{eq:int:s:i} by $2e^{\pi i s / 2}$ to
  obtain \eqref{eq:int:s}.
  Note that the integrals are absolutely convergent for $s \in ( - 1, 1)$.
  By Dirichlet's test, \eqref{eq:int:s:i} holds for
  $s \in ( - 1, 2)$.
\end{proof}

Replacing $s$ with $s + 1$ in \eqref{eq:int:s}, we obtain the following
companion integral.

\begin{Corollary}
  For $s \in (- 2, 1)$ and $n \geq 0$,
  \begin{gather}
    \frac{\pi}{2} e^{- n} = \int_0^{\infty} \frac{x \sin (n x - \pi s /
    2)}{x^2 + 1} x^s \mathd x. \label{eq:int:s:sin}
  \end{gather}
\end{Corollary}

\begin{Example}
  Setting $s = 0$ in \eqref{eq:int:s:sin}, we f\/ind that
  \begin{gather}
    \frac{\pi}{2} e^{- n} = \int_0^{\infty} \frac{x \sin (n x)}{x^2 + 1}
    \mathd x, \label{eq:int:sin}
  \end{gather}
which is well known.  After taking one derivative with respect to $s$ in
\eqref{eq:int:s:sin}
and setting $s=0$, we similarly f\/ind that
  \begin{gather}\label{dd}
    0 = \int_0^{\infty} \frac{x \sin (n x)}{x^2 + 1} \log x\, \mathd x -
     \frac{\pi}{2} \int_0^{\infty} \frac{x \cos (n x)}{x^2 + 1} \mathd x,
  \end{gather}
  which may be compared with Ramanujan's formula \eqref{4}. As a
  second example, after taking two derivatives of \eqref{eq:int:s:sin} with
  respect to $s$, setting $s=0$, and using
  \eqref{eq:int:sin}, we arrive at the identity
  \begin{gather}\label{cc}
    \frac{\pi^3}{8} e^{- n} = \int_0^{\infty} \frac{x \sin (n x)}{x^2 + 1}
     \log^2 x\, \mathd x - \pi \int_0^{\infty} \frac{x \cos (n x)}{x^2 + 1}
     \log  x \,\mathd x.
  \end{gather}
\end{Example}

  We of\/fer a few additional remarks before generalizing our ideas in the
  next section.  Equating real
  parts in the identity~\eqref{eq:int:s:i} from the proof of
  Theorem~\ref{thm:int:s}, we f\/ind that
  \begin{gather}
    \pi e^{- n} \cos (\pi s / 2) = \int_0^{\infty} \big( \cos (n x) (1 + \cos (\pi
    s)) + \sin (n x) \sin (\pi s)\big) \frac{x^s }{x^2 + 1}\mathd x.
    \label{eq:int:s:re}
  \end{gather}
  Setting $s = 0$ in \eqref{eq:int:s:re}, we again obtain  \eqref{easy}.
  On the other hand, note that
  \begin{gather*}
    \left[ \frac{\mathd}{\mathd s} \big( \cos (n x) (1 + \cos (\pi s)) + \sin (n
     x) \sin (\pi s)\big) \right]_{s = 0} = \pi \sin (n x) .
  \end{gather*}
  Hence, taking a derivative of \eqref{eq:int:s:re} with respect to $s$, and
  then setting $s = 0$, we f\/ind that
  \begin{gather*}
    0 = \pi \int_0^{\infty} \frac{\sin (n x)}{x^2 + 1} \mathd x + 2
     \int_0^{\infty} \frac{\cos (n x)}{x^2 + 1} \log x\, \mathd x,
  \end{gather*}
  which is the formula \eqref{4} that Ramanujan recorded on p.~391. Similarly, taking
  two derivatives of \eqref{eq:int:s:re} and letting $s=0$, we deduce that
  \begin{gather*}
    - \frac{\pi^3}{4} e^{- n} = - \pi^2 \int_0^{\infty} \frac{\cos (n x)}{x^2
     + 1} \mathd x + 2 \pi \int_0^{\infty} \frac{\sin (n x)}{x^2 + 1} \log x\,
     \mathd x + 2 \int_0^{\infty} \frac{\cos (n x)}{x^2 + 1} \log^2x\, \mathd
     x,
  \end{gather*}
  which, using \eqref{easy}, simplif\/ies to
  \begin{gather*}
    \frac{\pi^3}{8} e^{- n} = \pi \int_0^{\infty} \frac{\sin (n x)}{x^2 + 1}
     \log x \, \mathd x + \int_0^{\infty} \frac{\cos (n x)}{x^2 + 1} \log^2x\,
     \mathd x
  \end{gather*}
  which is the formula \eqref{4a} arising from Ramanujan's unintelligible remark in the initial edition of~\cite{nb}.

  The integral \eqref{eq:int:s:re} has the companion
  \begin{gather}
    \pi e^{- n} \sin (\pi s / 2) = \int_0^{\infty} \big( \cos (n x) \sin (\pi s) +
    \sin (n x) (1 - \cos (\pi s))\big) \frac{x^s }{x^2 + 1}\mathd x,
    \label{eq:int:s:im}
  \end{gather}
  which is obtained by equating imaginary parts in~\eqref{eq:int:s:i}.
  However, taking derivatives of~\eqref{eq:int:s:im} with respect to~$s$, and
  then setting $s = 0$, does not generate new identities. Instead, we recover
  precisely the previous results. For instance, taking a derivative of~\eqref{eq:int:s:im} with respect to~$s$, and then setting $s = 0$, we again
  deduce~\eqref{easy}.
  Taking two derivatives of~\eqref{eq:int:s:im} with
  respect to~$s$, and then setting $s = 0$, we obtain
  \begin{gather*}
    0 = \pi^2 \int_0^{\infty} \frac{\sin (n x)}{x^2 + 1} \mathd x + 2 \pi
     \int_0^{\infty} \frac{\cos (n x)}{x^2 + 1} \log x\, \mathd x,
  \end{gather*}
  which is again Ramanujan's formula~\eqref{4}.

\section{General theorems}

The phenomenon observed by Ramanujan in \eqref{2}
can be generalized by replacing the rational function $1/(z^2+1)$ by a
general rational function $f(z)$ in which the denominator has degree at least
one greater than the degree of the numerator.  We shall also assume that
$f(z)$ does not have any poles on the real axis.  We could prove a theorem
allowing for poles on the real axis, but in such instances we would need to
consider the principal values of the resulting integrals on the real axis.
In our arguments above, we used the fact that $1/(z^2+1)$ is an even
function.  For our general theorem, we require that $f(z)$ be either even or
odd.  For brevity, we let $\text{Res}(F(z);z_0)$ denote the residue of a~function~$F(z)$ at a pole~$z_0$.  As above, we def\/ine  a branch of~$\log z$
by $-\frac12\pi<\theta=\arg z\leq \frac32\pi$.

For a rational function $f(z)$ as prescribed above and each
nonnegative integer $m$, def\/ine
\begin{gather}\label{IJ}
I_m:=\int_0^{\infty}f(x) \cos x \log^mx\,\mathd x \qquad\text{and}\qquad J_m:=\int_0^{\infty}f(x) \sin x \log^mx\,
\mathd x.
\end{gather}

\begin{Theorem}\label{theorem2} Let $f(z)$ denote a rational function in $z$, as described above, and let $I_m$ and $J_m$ be defined by~\eqref{IJ}.  Let
\begin{gather}\label{S}
 S:=2\pi i\sum_U\textup{Res}(e^{iz}f(z)\log^mz,z_j),
\end{gather}
where the sum is over all poles $z_j$ of $e^{iz}f(z)\log^mz$ lying in the upper half-plane~$U$.  Suppose that~$f(z)$ is even.  Then
\begin{gather}\label{S1}
S=\sum_{k=0}^m\binom{m}{k}(i\pi)^{m-k}(I_k-iJ_k)+(I_m+iJ_m).
\end{gather}
Suppose that $f(z)$ is odd.  Then
\begin{gather}\label{S2}
S=-\sum_{k=0}^m\binom{m}{k}(i\pi)^{m-k}(I_k-iJ_k)+(I_m+iJ_m).
\end{gather}
\end{Theorem}

Observe that \eqref{S1} and \eqref{S2} are recurrence relations that enable
us to successively calculate~$I_m$ and~$J_m$.  With each succeeding value of~$m$, we see that two previously non-appearing integrals arise.  If~$f(z)$ is
even, then these integrals are $I_m$ and $J_{m-1}$, while if $f(z)$ is odd,
these integrals are $J_m$ and $I_{m-1}$.  The previously non-appearing
integrals appear in either the real part or the imaginary part of the
right-hand sides of~\eqref{S1} and~\eqref{S2}, but not both real and
imaginary parts.  This fact therefore does not enable us to explicitly
determine either of the two integrals.  We must be satisf\/ied with obtaining
recurrence relations with increasingly more terms.

\begin{proof}
We commence as in the proof of Theorem~\ref{theorem1}.  Let
  $C_{R,\varepsilon}$ denote the positively oriented contour consisting of
  the semi-circle~$C_{R}$ given by  $z=Re^{i\theta}$, $0\leq\theta\leq\pi$,
  $[-R,-\epsilon]$, the semi-circle~$C_{\varepsilon}$ given by $z=\varepsilon
  e^{i\theta}$, $\pi\geq\theta\geq0$, and $[\varepsilon,R]$, where
  $0<\varepsilon<d$, where~$d$ is the smallest modulus of the poles of~$f(z)$
  in~$U$.  We also choose~$R$ larger than the moduli  of all the poles of~$f(z)$ in~$U$. By the residue theorem,
\begin{gather}\label{1a}
\int_{C_{R,\varepsilon}}e^{iz}f(z)\log^mz\,\mathd z=S,
\end{gather}
where $S$ is def\/ined in~\eqref{S}.

We next directly evaluate the integral on the left-hand side of~\eqref{1a}.
As in the proof of Theorem~\ref{theorem1}, we can easily show that
\begin{gather}\label{2a}
\int_{C_{\varepsilon}}e^{iz}f(z)\log^mz\,\mathd z=o(1),
\end{gather}
as $\varepsilon$ tends to $0$.
Secondly, we estimate the integral over $C_R$.  By hypothesis, there exist a~positive constant $A$ and a~positive number~$R_0$, such that for $R\geq R_0$,
$|f(Re^{i\theta})|\leq A/R$.  Hence, for $R\geq R_0$,
\begin{gather}
\left|\int_{C_{R}}e^{iz}f(z)\log^mz\,\mathd z\right| =
\left|\int_0^\pi
  e^{iRe^{i\theta}}f(Re^{i\theta})\log^m(Re^{i\theta})iRe^{i\theta}\mathd\theta\right|
\nonumber\\
\hphantom{\left|\int_{C_{R}}e^{iz}f(z)\log^mz\,\mathd z\right|}{}
\leq\int_0^\pi e^{-R\sin\theta}|f(Re^{i\theta})|(\log R+\pi)^mR\,\mathd\theta\nonumber\\
\hphantom{\left|\int_{C_{R}}e^{iz}f(z)\log^mz\,\mathd z\right|}{}
 \leq A(\log R+\pi)^m\left(\int_0^{\pi/2}+\int_{\pi/2}^{\pi}\right)e^{-R\sin\theta}\mathd\theta.\label{3a}
\end{gather}
Since $\sin\theta\geq2\theta/\pi$, $0\leq\theta\leq\pi/2$, upon replacing $\theta$ by $\pi-\theta$, we f\/ind that
\begin{align}\label{4ab}
\int_{\pi/2}^{\pi}e^{-R\sin\theta}\mathd\theta&=\int_0^{\pi/2}e^{-R\sin\theta}\mathd\theta
\leq\int_0^{\pi/2}e^{-2R\theta/\pi}\mathd\theta
=\frac{\pi}{2R}\left(1-e^{-R}\right).
\end{align}
The bound~\eqref{4ab} also holds for the f\/irst integral on the far right-hand side of~\eqref{3a}.  Hence, from~\eqref{3a},
\begin{gather}\label{5a}
\left|\int_{C_{R}}e^{iz}f(z)\log^mz\,\mathd z\right|\leq A(\log R+\pi)^m\frac{\pi}{R}\left(1-e^{-R}\right)=o(1),
\end{gather}
as $R$ tends to inf\/inity.

Hence, so far, by \eqref{1a}, \eqref{2a}, and \eqref{5a}, we have shown that
\begin{gather}
S=\int_{-\infty}^0e^{ix}f(x)(\log|x|+i\pi)^m\mathd x+\int_0^{\infty}e^{ix}f(x)\log^mx\,\mathd x\nonumber\\
\hphantom{S}{}=\int_0^{\infty}\left\{e^{-ix}f(-x)(\log x+i\pi)^m+e^{ix}f(x)\log^mx\right\}\,\mathd x.\label{6a}
\end{gather}
Suppose f\/irst that $f(x)$ is even.  Then \eqref{6a} takes the form
\begin{gather*}
S =\int_0^{\infty}f(x)\left\{e^{-ix}(\log x+i\pi)^m+e^{ix}\log^mx\right\}\mathd x\nonumber\\
\hphantom{S}{} =\int_0^{\infty}f(x)\left\{e^{-ix}\sum_{k=0}^m\binom{m}{k}\log^kx(i\pi)^{m-k}+e^{ix}\log^mx\right\}\mathd x\nonumber\\
\hphantom{S}{} =\sum_{k=0}^m\binom{m}{k}(i\pi)^{m-k}(I_k-iJ_k)+(I_m+iJ_m),
\end{gather*}
which establishes \eqref{S1}.  Secondly, suppose that~$f(z)$ is odd.  Then, \eqref{6a} takes the form
\begin{gather}
S =\int_0^{\infty}f(x)\left\{-e^{-ix}(\log x+i\pi)^m+e^{ix}\log^mx\right\}\mathd x\nonumber\\
\hphantom{S}{} =\int_0^{\infty}f(x)\left\{-e^{-ix}\sum_{k=0}^m\binom{m}{k}\log^kx(i\pi)^{m-k}+e^{ix}\log^mx\right\}\mathd x\nonumber\\
\hphantom{S}{}=-\sum_{k=0}^m\binom{m}{k}(i\pi)^{m-k}(I_k-iJ_k)+(I_m+iJ_m),\label{8a}
\end{gather}
from which \eqref{S2} follows.
\end{proof}

\begin{Example}  Let $f(z)=z/(z^2+1)$. Then
\begin{gather*}
2\pi i\,\textup{Res}\left(\frac{e^{iz}z\log^mz}{z^2+1},i\right)=\frac{\pi i}{e}\left(\frac{i\pi}{2}\right)^m,
\end{gather*}
and so we are led by \eqref{S2} to the recurrence relation
\begin{gather}\label{e2}
\frac{\pi
  i}{e}\left(\frac{i\pi}{2}\right)^m=-\sum_{k=0}^m\binom{m}{k}(i\pi)^{m-k}(I_k-iJ_k)+(I_m+iJ_m),
\end{gather}
where
\begin{gather*}
I_m:=\int_0^{\infty}\frac{x \cos x \log^mx}{x^2+1}\mathd x \quad\text{and}\quad J_m:=\int_0^{\infty}\frac{x \sin x \log^mx}{x^2+1}
\mathd x.
\end{gather*}
(In the sequel, it is understood that we are assuming that $n=1$ in Theorem~\ref{theorem1} and in all our deliberations of the two preceding sections.)
If $m=0$, \eqref{e2} reduces to
 \begin{gather}\label{e3}
 J_0=\frac{\pi}{2e},
 \end{gather}
 which is \eqref{easy}. After simplif\/ication, if $m=1$, \eqref{e2} yields
 \begin{gather}\label{e4}
 -\frac{\pi^2}{2e}=-i\pi I_0-\pi J_0+2iJ_1.
 \end{gather}
 If we equate real parts in \eqref{e4}, we once again deduce~\eqref{e3}.  If we equate imaginary parts in~\eqref{e4}, we f\/ind that
 \begin{gather}\label{e5}
 J_1-\frac{\pi}{2}I_0=0,
 \end{gather}
 which is identical with~\eqref{dd}.  Setting $m=2$ in~\eqref{e2}, we f\/ind that
 \begin{gather}\label{e6}
 -\frac{i\pi^3}{4e}=\pi^2(I_0-iJ_0)-2i\pi(I_1-iJ_1)+2iJ_2.
 \end{gather}
 Equating real parts on both sides of~\eqref{e6}, we once again deduce~\eqref{e5}. If we equate imaginary parts in~\eqref{e6} and employ~\eqref{e3}, we arrive at
 \begin{gather}\label{e7}
 J_2-\pi I_1=\frac{\pi^3}{8e},
 \end{gather}
 which is the same as~\eqref{cc}.
 Lastly, we set $m=3$ in~\eqref{e2} to f\/ind that
 \begin{gather}\label{e8}
 \frac{\pi^4}{8e}=i\pi^3(I_0-iJ_0)+3\pi^2(I_1-iJ_1)-3i\pi(I_2-iJ_2)+2iJ_3.
 \end{gather}
 If we equate real parts on both sides of~\eqref{e8} and simplify, we deduce~\eqref{e7} once again. On the other hand, when we equate imaginary parts on
 both sides of~\eqref{e8}, we deduce that
 \begin{gather}\label{e9}
 2J_3-3\pi I_2-3\pi^2J_1+\pi^3I_0=0.
 \end{gather}
 A slight simplif\/ication of~\eqref{e9} can be rendered with the use of~\eqref{e5}.
\end{Example}

 We can replace the rational function $1 / ( x^2 +
1)$ in Theorem~\ref{thm:int:s} by other even rational functions~$f ( x)$ to
obtain the following
generalization of Theorem~\ref{thm:int:s}.  Its proof is in the same spirit
as that of Theorem~\ref{theorem2}.

\begin{Theorem}
  \label{thm:int:f:s}Suppose that~$f (z)$ is an even rational function with
  no real poles and with the degree of the denominator exceeding the degree
  of the numerator by at least~$2$. Then,
  \begin{gather*}
    \frac{\pi i}{e^{\pi i s / 2}}  \sum_U \operatorname{Res} ( e^{i n z} f ( z) z^s,
     z_j) = \int_0^{\infty} \cos (n x - \pi s / 2) f ( x) x^s \mathd x,
  \end{gather*}
  where the sum is over all poles~$z_j$ of~$f ( z)$ lying in the upper
  half-plane~$U$.
\end{Theorem}

Note that, as we did for~\eqref{eq:int:s:sin}, we can replace~$s$ with $s + 1$
in Theorem~\ref{thm:int:f:s} to obtain a~corresponding result for odd rational
functions~$x f (x)$. This is illustrated in Example~\ref{eg:odd} below.

As an application, we derive from Theorem~\ref{thm:int:f:s} the following
explicit integral evaluation, which reduces to Theorem~\ref{thm:int:s} when $r
= 0$.

\begin{Theorem}
  \label{thm:int:sr}Let $r \geq 0$ be an integer. For $s \in (- 1, 2 ( r
  + 1))$ and $n \geq 0$,
  \begin{gather*}
    \int_0^{\infty} \frac{\cos (n x - \pi s / 2)}{( x^2 + 1)^{r + 1}} x^s
     \mathd x = \frac{\pi}{2} e^{- n} \sum_{k = 0}^r \frac{1}{2^{r + k}}
     \binom{r + k}{k} \sum_{j = 0}^{r - k} ( - 1)^j \binom{s}{j} \frac{n^{r -
     k - j}}{( r - k - j) !} .
  \end{gather*}
\end{Theorem}

\begin{proof}
  Setting $f ( z) = 1 / ( z^2 + 1)^r$ in Theorem~\ref{thm:int:f:s}, we see that we need  to
  calculate the residue
  \begin{gather*}
    \operatorname{Res} \left( \frac{e^{i n z} z^s}{( z^2 + 1)^{r + 1}}, i \right) =
     \operatorname{Res} \left( \frac{\alpha (z)}{( z - i)^{r + 1}}, i \right),
  \end{gather*}
  where
  \begin{gather*}
    \alpha (z) = \frac{e^{i n z} z^s}{(z + i)^{r + 1}}
  \end{gather*}
  is analytic in a neighborhood of $z = i$. Equivalently, we calculate the
  coef\/f\/icient of $x^r$ in the Taylor expansion of $\alpha (x + i)$ around $x =
  0$. Using the binomial series
  \begin{gather*}
    \frac{1}{( x + a)^{r + 1}} = \sum_{k \geq 0} ( - 1)^k \binom{r +
     k}{k} x^k a^{- r - k - 1}
  \end{gather*}
  with $a = 2 i$, we f\/ind that
  \begin{gather*}
    \alpha (x + i)   =  e^{- n}  \frac{e^{i n x} (x + i)^s}{(x + 2 i)^{r +
    1}} \\
    \hphantom{\alpha (x + i)}{}
      =  e^{- n} \sum_{k \geq 0} ( - 1)^k \binom{r + k}{k} x^k (2 i)^{-
    r - k - 1} \sum_{j \geq 0} \binom{s}{j} x^j i^{s - j} \sum_{l
    \geq 0} \frac{(i n x)^l}{l!} .
  \end{gather*}
  Extracting the coef\/f\/icient of $x^r$, we conclude that
  \begin{gather*}
    \operatorname{Res} \left( \frac{e^{i n z} z^s}{( z^2 + 1)^{r + 1}}, i \right)   =
     \frac{e^{- n}}{( 2 i)^{r + 1}} \sum_{k = 0}^r \frac{( - 1)^k}{( 2 i)^k}
    \binom{r + k}{k} \sum_{j = 0}^{r - k} \binom{s}{j} i^{s - j} \frac{( i
    n)^{r - k - j}}{( r - k - j) !}\\
  \hphantom{\operatorname{Res} \left( \frac{e^{i n z} z^s}{( z^2 + 1)^{r + 1}}, i \right)}{}
    =  \frac{e^{- n} e^{\pi i s / 2}}{2 i} \sum_{k = 0}^r \frac{1}{2^{r +
    k}} \binom{r + k}{k} \sum_{j = 0}^{r - k} ( - 1)^j \binom{s}{j} \frac{n^{r
    - k - j}}{( r - k - j) !}.
  \end{gather*}
  Theorem \ref{thm:int:sr} now follows from Theorem~\ref{thm:int:f:s}.
\end{proof}

\begin{Example}
  In particular, in the case $s = 0$,
  \begin{gather}\label{ab}
    \int_0^{\infty} \frac{\cos (n x)}{( x^2 + 1)^{r + 1}} \mathd x =
     \frac{\pi}{2} e^{- n} \sum_{k = 0}^r \frac{1}{2^{r + k}} \binom{r + k}{k}
     \frac{n^{r - k}}{( r - k) !}.
  \end{gather}
  We note that, more generally, this integral can be expressed in terms of the modif\/ied Bessel function $K_{r+1/2}(z)$ of order $r+1/2$.
   Namely, we have \cite[formula~(3.773), no.~6]{gr}
   \begin{gather}\label{ab2}
    \int_0^{\infty} \frac{\cos (n x)}{( x^2 + 1)^{r + 1}} \mathd x = \left(
     \frac{n}{2} \right)^{r+1/2}  \frac{\sqrt{\pi}}{\Gamma(r+1)} K_{r+1/2} (n) .
  \end{gather}
  When $r\ge0$ is an integer, the Bessel function $K_{r+1/2}(z)$ is elementary and the right-hand side of~\eqref{ab2} evaluates to the right-hand side of~\eqref{ab}.
\end{Example}

On the other hand, taking a derivative with respect to $s$ before setting $s =
0$, and observing that, for $j \geq 1$,
\begin{gather*}
  \left[ \frac{\mathd}{\mathd s} \binom{s}{j} \right]_{s = 0} = \frac{(-
   1)^{j - 1}}{j},
\end{gather*}
we arrive at the following generalization of Ramanujan's formula
\eqref{4}.

\begin{Corollary}
  We have
  \begin{gather*}
      \int_0^{\infty} \frac{\sin (n x)}{( x^2 + 1)^{r + 1}} \mathd x
    + \frac{2}{\pi} \int_0^{\infty} \frac{\cos (n x)}{( x^2 + 1)^{r + 1}} \log  x\, \mathd x\\
    \qquad{} =   - e^{- n} \sum_{k = 0}^r \frac{1}{2^{r + k}} \binom{r
    + k}{k} \sum_{j = 1}^{r - k} \frac{1}{j}  \frac{n^{r - k - j}}{( r - k -
    j) !} .
  \end{gather*}
\end{Corollary}

We leave it to the interested reader to make explicit the corresponding
generalization of~\eqref{eq:int:rama2}.

\begin{Example}  \label{eg:odd}
  As a direct extension of~\eqref{eq:int:s:sin},  replacing
  $s$ with $s + 1$ in Theorem~\ref{thm:int:sr}, we obtain the following
  companion integral. For integers $r \geq 0$, and any $s \in (- 2, 2 r +
  1)$ and $n \geq 0$,
  \begin{gather*}
    \int_0^{\infty} \frac{x \sin (n x - \pi s / 2)}{( x^2 + 1)^{r + 1}} x^s
     \mathd x = \frac{\pi}{2} e^{- n} \sum_{k = 0}^r \frac{1}{2^{r + k}}
     \binom{r + k}{k} \sum_{j = 0}^{r - k} ( - 1)^j \binom{s + 1}{j}
     \frac{n^{r - k - j}}{( r - k - j) !} .
  \end{gather*}
  In particular, setting $s = 0$, we f\/ind that
  \begin{gather}\label{abc}
    \int_0^{\infty} \frac{x \sin (n x)}{( x^2 + 1)^{r + 1}} \mathd x =
     \frac{\pi}{2} e^{- n} \sum_{k = 0}^r \frac{1}{2^{r + k}} \binom{r + k}{k}
     \left\{ \frac{n^{r - k}}{( r - k) !} - \frac{n^{r - k - 1}}{( r - k - 1)
     !} \right\},
  \end{gather}
  while taking a derivative with respect to $s$ before setting $s = 0$ and
  observing that, for~$j \geq 2$,
  \begin{gather*}
    \left[ \frac{\mathd}{\mathd s} \binom{s + 1}{j} \right]_{s = 0} =
     \frac{(- 1)^j}{j (j - 1)},
  \end{gather*}
  we f\/ind that{\samepage
  \begin{gather*}
      \int_0^{\infty} \frac{x \cos (n x)}{( x^2 + 1)^{r + 1}} \mathd x
    - \frac{2}{\pi} \int_0^{\infty} \frac{x \sin (n x)}{( x^2 + 1)^{r + 1}} \log x \, \mathd x\\
   \qquad{} =   e^{- n} \sum_{k = 0}^r \frac{1}{2^{r + k}} \binom{r +
    k}{k} \left[ \frac{n^{r - k - 1}}{( r - k - 1) !} - \sum_{j = 2}^{r - k}
    \frac{1}{j (j - 1)} \frac{n^{r - k - j}}{( r - k - j) !} \right]\\
   \qquad{} =   \frac{2}{\pi} \int_0^{\infty} \frac{\cos (n x)}{( x^2 + 1)^{r + 1}} \mathd x
    - \frac{2}{\pi} \int_0^{\infty} \frac{x \sin (n x)}{( x^2 + 1)^{r + 1}} \mathd x\\
    \qquad\quad{} - e^{- n} \sum_{k = 0}^r \frac{1}{2^{r + k}} \binom{r + k}{k} \sum_{j = 2}^{r - k} \frac{1}{j (j - 1)} \frac{n^{r - k - j}}{( r - k - j) !},
  \end{gather*}
upon the employment of \eqref{ab}} and~\eqref{abc}.
\end{Example}

\subsection*{Acknowledgements}

We wish to thank Khristo Boyadzhiev, Larry Glasser and the referees for their careful and helpful suggestions.

\pdfbookmark[1]{References}{ref}
\LastPageEnding

\end{document}